\documentclass[11pt]{amsart}

\usepackage{amsmath,amssymb,latexsym,soul,cite,mathrsfs}

\usepackage{color,enumitem,graphicx}
\usepackage[colorlinks=true,urlcolor=blue,
citecolor=red,linkcolor=blue,linktocpage,pdfpagelabels,
bookmarksnumbered,bookmarksopen]{hyperref}
\usepackage[english]{babel}

\usepackage[left=2.9cm,right=2.9cm,top=2.8cm,bottom=2.8cm]{geometry}
\usepackage[hyperpageref]{backref}

\usepackage[colorinlistoftodos]{todonotes}
\makeatletter
\providecommand\@dotsep{5}
\def\listtodoname{List of Todos}
\def\listoftodos{\@starttoc{tdo}\listtodoname}
\makeatother

\numberwithin{equation}{section}

\newtheorem{theorem}{Theorem}[section]
\newtheorem{proposition}[theorem]{Proposition}
\newtheorem{lemma}[theorem]{Lemma}
\newtheorem{corollary}[theorem]{Corollary}

\newtheorem{claim}[theorem]{Claim}

\begin{document}

\title[On existence of multiple normalized solutions to a class ...]
{On existence of multiple normalized solutions to a class of elliptic problems in whole $\mathbb{R}^N$ via penalization method }
\author{Claudianor O. Alves$^*$ and Nguyen Van Thin}
\address[Claudianor O. Alves ]
{\newline\indent Unidade Acad\^emica de Matem\'atica
\newline\indent 
Universidade Federal de Campina Grande 
\newline\indent
e-mail: coalves@mat.ufcg.edu.br
\newline\indent
58429-970, Campina Grande - PB, Brazil} 

\address[Nguyen Van Thin]
{\newline\indent Department of Mathematics, Thai Nguyen University of Education, 
\newline\indent	Luong Ngoc Quyen street, Thai Nguyen city, Thai Nguyen, Viet Nam.
\newline\indent \mbox{}
\newline\indent Thang Long Institute of Mathematics and Applied Sciences, 
\newline\indent Thang Long University, Nghiem Xuan Yem, Hoang Mai, Hanoi, Viet Nam.
\newline\indent e-mail:thinmath@gmail.com and thinnv@tnue.edu.vn }

\pretolerance10000

\begin{abstract}
In this paper  we study the existence of multiple normalized solutions  to the following class of elliptic problems  
\begin{align*}
	\left\{
	\begin{aligned}
		&-\epsilon^2\Delta u+V(x)u=\lambda u+f(u), \quad
		\quad
		\hbox{in }\mathbb{R}^N,\\
		&\int_{\mathbb{R}^{N}}|u|^{2}dx=a^{2}\epsilon^N,
	\end{aligned}
	\right.
\end{align*}
where $a,\epsilon>0$, $\lambda\in \mathbb{R}$ is an unknown parameter that appears as a Lagrange multiplier, 
$V:\mathbb{R}^N \to [0,\infty)$ is a continuous function, and $f$ is a continuous function with $L^2$-subcritical growth. It is
 proved that the number of normalized solutions is related to the topological richness of the set where the potential $V$
attains its minimum value. In the proof of our main result, we apply minimization techniques, Lusternik-Schnirelmann category and the penalization method due to del Pino and Felmer \cite{delpinoFelmer}. 
\end{abstract}

\thanks{ C. O. Alves  is the corresponding author and he was partially
supported by  CNPq/Brazil 307045/2021-8, Projeto Universal FAPESQ 3031/2021. Nguyen Van Thin is supported
by Ministry of Education and Training of Vietnam under project with
the name “Some properties about solutions to differential
equations, fractional partial differential equations” and grant
number B2023-TNA-14.}

\maketitle

\section{Introduction}

This paper is concerned with the existence of multiple normalized solutions to  the following class of elliptic problems 
\begin{align}\label{11}
	\left\{
	\begin{aligned}
		&-\epsilon^2\Delta u+V(x)u=\lambda u+f(u), \quad
		\hbox{in }\mathbb{R}^N,\\
		&\int_{\mathbb{R}^{N}}|u|^{2}dx=a^{2}\epsilon^N,
	\end{aligned}
	\right.
\end{align}
where $a,\epsilon>0$ and $\lambda\in \mathbb{R}$ is an unknown parameter that appears as a Lagrange multiplier. When $V(x)\equiv 0$ and $\epsilon=1$, 
problem (\ref{11}) becomes 
 \begin{equation}\label{eq0}
	\begin{cases}
		&-\Delta u=\lambda u+f(u), \quad
		\hbox{in }\mathbb{R}^N,\\
		&\displaystyle \int_{\mathbb{R}^{N}}|u|^{2}dx=a^{2},
	\end{cases}
\end{equation}
and there is a vast literature associated with that problem, see for example \cite{AJM2, CCM, valerio, Bartosz, CL, CingolaniJeanjean, Jun, Ikoma, jeanjean1, JeanjeanLu, JeanjeanLu2020, JeanjenaLu2, JeanjeanLe, JeanjeanLe2, Sh1, Nicola1, Nicola2, WW} and the references therein. In these works, the authors assume that the nonlinearity has $L^2$-subcritical, $L^2$-critical or $L^2$-super-critical growth. In $L^2$-subcritical growth case, the authors used the minimizing method to get a solution. For the other cases, the energy function is not bounded from below, and some authors explored the properties of the Pohozaev manifold to get the existence of normalized solutions.

 In the problem (\ref{11}), if $0\not\equiv V(x)\le 0$ as $|x|\to\infty$ and $\epsilon=1,$ then Ikoma and Miyamoto \cite{IkomaMiyamoto} studied the existence
 and nonexistence of a minimizer of the $L^2$-constrain minimization problem 
$$
e(a)=\inf\{E(u)| u\in H^1(\mathbb R^N), |u|_2^2=a\},
$$
where 
$$
E(u)=\dfrac{1}{2}\int\limits_{\mathbb R^N}(|\nabla u|^2dx+V(x)|u|^2)dx-\int\limits_{\mathbb R^N}F(u)dx,
$$
with $V$ and $f$ satisfying  some suitable assumptions. Since $V\in L^{\infty}(\mathbb R^N)$, then $E$ is not  weakly lower semicontinuous and the scaled transformation
 $x\mapsto u(\lambda x)$ is not useful in their case. Therefore, it was necessary to perform a careful interaction estimate to exclude dichotomy by using a scheme due to Shibata in \cite{Sh1}. In \cite{DZ}, Ding and Zhong  proved the existence of solution 
 $(u,\lambda)\in H^1(\mathbb R^N)\times \mathbb R$ for the problem 
 \begin{equation*}\label{pt1}
 	\begin{cases}
 		&-\Delta u+V(x)u+\lambda u=g(u), \quad
 		\quad
 		\hbox{in }\mathbb{R}^N, N\ge 3,\\
 		&u\ge 0, \displaystyle \int_{\mathbb{R}^{N}}|u|^{2}dx=a^{2},
 	\end{cases}
 \end{equation*}
 where the nonlinearity $g$ and potential function $V$ satisfy some suitable assumptions. They treated the mass super-critical case, still using some good property of Pohozaev manifold and mini-max structure. In 2020, Chen and Tang \cite{CT} studied the existence of normalized ground state solutions for the following nonautonomous Schr\"odinger equation
 \begin{equation*}\label{pt2}
 	\begin{cases}
 		&-\Delta u-a(x)f(u)=\lambda u, x\in\mathbb R^N\quad\\
 		&u\in H^1(\mathbb R^N) , \displaystyle \int_{\mathbb{R}^{N}}|u|^{2}dx=a^{2},
 	\end{cases}
 \end{equation*}
where $N\ge 1, \lambda\in\mathbb R, a\in C(\mathbb R^N,[0,+\infty))$ with $0<a_{\infty}=\displaystyle \lim_{|y|\to\infty}a(y)\le a(x)$ and $f\in C(\mathbb R,\mathbb R)$ satisfies some general assumptions. In that paper, the authors showed a new estimate to recover the compactness for a minimizing sequence on a suitable manifold as follows:
\begin{align*}
\mathcal M(a)=\left\{u\in S(a): J(u)=\dfrac{d}{dt}I(t^{N/2}u_t)\Big|_{t=1}=0\right\},
\end{align*}
where 
\begin{equation} \label{Sa2}
	S(a)=\{u \in H^{1}(\mathbb{R}^N) \,:\, | u |_2=a\, \},
\end{equation} 
$$
I(u)=\dfrac{1}{2}\int\limits_{\mathbb R^N}|\nabla u|^2dx-\int\limits_{\mathbb R^N}a(x)F(u)dx, 
$$
$F(t)=\int_{0}^{t}f(s)\,ds$, $u_t(x)=u(tx)$ for all $t>0$ and $u\in H^1(\mathbb R^N).$ In 2021, Bartsch, Molle, Rizzi and Verzini \cite{Bartschmolle} studied the existence of a solution for  
\begin{align}\label{pt3}
	-\Delta u+\lambda u-V(x)u=|u|^{p-2}u,\;u\in H^1(\mathbb R^N),\lambda\in \mathbb R, 
\end{align}
 by assuming that $V(x)\ge 0$, $V(x)\to 0$ as $|x|\to \infty$, $2+\dfrac{4}{N}<p<2^{*}, 2^{*}=\dfrac{2N}{N-2}$ if $N\ge 3$ and $2^{*}=+\infty$ if $N=1,2,$ and some technical conditions involving the potential $V$ and the mass $a$. In that paper, the energy functional has a mountain pass structure, but the mountain pass value is the
 mountain pass value for the case $V \equiv 0$ and it is not achieved. In order to overcome this difficulty the authors presented a new linking argument for the energy functional on $S(a)$. After that, in 2022, Molle, Riey and Verzini \cite{MRV} proved the  existence of positive solution with prescribed $L^2$-norm for the problem mentioned above with $p\in (2+\dfrac{4}{N},2^{*}).$ Under a smallness assumption on $V$ and no condition on the mass, they were able to establish the existence of a mountain pass solution with positive energy and negative energy. If the mass is small enough with suitable level, they also showed the existence of solution to (\ref{pt3}). The method used is based on Splitting Lemma to get the compactness result of Palais-Smale sequence. Still in 2022, Yang, Qi and Zou \cite{YQZ} studied the existence and multiplicity of normalized solutions to the following
Schr\"odinger equations with potentials and non-autonomous nonlinearities:

 \begin{equation}\label{pt4}
 	\begin{cases}
 		&-\Delta u+V(x)u+\lambda u=f(x,u), x\in\mathbb R^N\quad\\
 		&u\in H^1(\mathbb R^N) , \displaystyle \int_{\mathbb{R}^{N}}|u|^{2}dx=a,
 	\end{cases}
 \end{equation}
where $V(x)\le \displaystyle \lim_{|x|\to\infty}V(x):=V_{\infty}\in (-\infty,+\infty]$ and $f(x,s)$ satisfies the Berestycki-
Lions type conditions with mass sub-critical growth. In the case $V_{\infty}=+\infty,$ it was proved that the problem (\ref{pt4}) has a ground state solution for all $a>0.$ Furthermore, if $f$ is an odd function with respect to the second variable, then (\ref{pt4}) has infinitely many normalized solutions with increasing energy. When $V_{\infty}<+\infty,$ it was showed that there exists $a_0\ge 0$ such that problem (\ref{pt4}) has a solution for all $a\ge a_0.$ Finally, they investigated the multiplicity of normalized radial solutions by index theory and genus method.

We note that, when $\epsilon$ is small enough and $V\not\equiv 0,$ there are very few results related to the problem (\ref{11}). In 2022, Alves-Ji \cite{AJ21} proved the existence of solution for (\ref{11})
for $f(u)=|u|^{q-2}u, q\in (2,2+\dfrac{4}{N}), N\ge 2,$  where the potential $V$ is bounded continuous function, $1$-periodic in $x_1,\dots,x_N,$ and satisfies some suitable assumptions.  In order to prove their results, the authors proved  a compactness result for minimizing sequences $(u_n)$ constrained on $S(a)$. In a recent paper, without using the genus theory, Alves \cite{A21} studied the existence of multiple solutions for a  problem of the type 
\begin{equation*}
	\begin{cases}
		&-\Delta u=\lambda u+h(\epsilon x)f(u), \quad
		\quad
		\hbox{in }\mathbb{R}^N,\\
		&u\ge 0,\displaystyle  \int_{\mathbb{R}^{N}}|u|^{2}dx=a^{2},
	\end{cases}
\end{equation*}
where $a>0,\epsilon>0,$ $\lambda \in \mathbb R$ is an unknown parameter, $h:\mathbb R^N\to [0,+\infty)$ is a continuous function and 
$f$ is a continuous function with $L^2$-subcritial growth. In that paper, the author showed that the number of normalized solutions is at least the 
number of global maximum points of $h$ when  $\epsilon$ is small enough. In 2022, Zhang-Zhang \cite{ZZ} investigated  the existence and concentration behaviour of the multi-bump solutions for the nonlinear Schr\"odinger equation
\begin{equation*}\label{pt4*}
	\begin{cases}
		& -h^2\Delta v-K(x)|v|^{2\sigma}v=-\lambda v, \quad
		\quad
		\hbox{in }\mathbb{R}^N,\\
		&|v|_2^2=  \int_{\mathbb{R}^{N}}|v|^{2}dx=m_0h^{\alpha},\; v(x)\to 0\;\text{as}\; |x|\to\infty,
	\end{cases}
\end{equation*}
where $h>0$ is a small parameter, $\alpha$ is a constant satisfying some suitable conditions and the nonlinearity can have $L^2$-subcritical or $L^2$-supercritical growth and the potential function $K>0$ possesses several local maximum points. By using the variational method, the authors construct normalized multi-bump solutions concentrating at finite set of local maximum points of $K.$ Since the problem is raised on the sphere in $L^2,$ then it is difficult to perform the minimization argument locally. To overcome this difficulty, they introduced a new penalized functional to identity the solution. In the same year, Alves-Thin \cite{AT1} showed the existence of multiple normalized solutions to problem (\ref{11}) via Lusternik-Schnirelmann category by assuming that the potential function $V$ satisfies the following conditions: \linebreak $V \in C(\mathbb{R}^N,\mathbb{R}) \cap L^{\infty}(\mathbb{R}^N)$, $V(0)=0$ and 
$$
0 = \inf_{x  \in \mathbb{R}^N}V(x)< \liminf_{|x| \to +\infty}V(x).
$$
That kind of hypothesis on potential function $V$ was introduced by Rabinowitz in \cite{R92} to prove the existence of solution for problems of the type
\begin{align} \label{PRab}
	\left\{
	\begin{aligned}
		&-\epsilon^2\Delta u+V(x)u=f(u), \quad
		\quad
		\hbox{in }\mathbb{R}^N,\\
		& u \in H^{1}(\mathbb{R}^N),
	\end{aligned}
	\right.
\end{align}
when $\epsilon$ is small enough. 

Motivate by a seminal paper due to Pino and Felmer \cite{delpinoFelmer} that studied the existence and concentration phenomena of solutions for problem (\ref{PRab}) by assuming that exists a bounded set $\Lambda \subset \mathbb{R}^N$ such that 
$$
\min_{x  \in \overline{\Lambda}}V(x)< \min_{x \in \partial \Lambda}V(x)=V_\infty, \leqno{(V_1)}
$$
it seems natural to study the existence of multiple normalized solution to problem (\ref{11}) when $V$ satisfies the condition $(V_1).$ Here, inspired by penalized method found in \cite{delpinoFelmer} and some arguments found in \cite{AFig2}, we will show the existence of multiple normalized solutions by using the Lusternik-Schnirelmann category, which is the main contribution of the present paper for this class of problem. We would like to point out that our penalization  is a little bit different from that explored in the reference above, see the beginning of Section 3.  Another important point that we would like to mention is that our approach is totally different from that explored in \cite{delpinoFelmer} and \cite{AFig2}, because in our case the energy functional associated with the problem (\ref{11}) does not satisfy the mountain pass geometry.

Hereafter, we suppose  that the nonlinearity $f$ is a continuous function with a $L^2$-subcritical growth and satisfies the  following assumptions:
\begin{itemize}
	\item [$(f_1)$] $f$ is odd and $\displaystyle \lim_{s \to 0}\frac{|f(s)|}{|s|^{q_0-1}} =\alpha>0$ for some $q_0 \in (2,2+\frac{4}{N})$;
	\item [$(f_2)$] There exist constants $c_1, c_2 > 0$ and $p_1 \in (2,2+\frac{4}{N})$ such that
	$$
	|f(s)| \leq c_1 + c_2|s|^{p_1-1}, \quad \forall s \in \mathbb{R};
	$$
	\item [$(f_3)$] There is $q \in [q_0,2+\frac{4}{N})$ such that $f(s)/s^{q-1}$ is an increasing function of $s$ on $(0,+\infty)$. 
\end{itemize}
We would like to point out that $(f_3)$ implies that $q \geq p_1$.

An example of a function $f$ that satisfies the above assumptions is 
$$
f(s)=|s|^{q_0-2}s+|s|^{p-2}s\ln(1+|s|), \quad \forall s \in \mathbb{R},
$$
for some $2<q_0<p<2+\frac{4}{N}<2^*$.  Here, $(f_3)$ occurs with $q=q_0$. 

Related to the function $V,$ we assume that the following condition  holds:
\begin{itemize}
\item[$(V)$]	$V \in C(\mathbb{R}^N,\mathbb{R}) \cap L^{\infty}(\mathbb{R}^N)$, $V(x) \geq 0$ for all $x \in\mathbb{R}^N$ and the condition $(V_1)$. Moreover, without lost of generality, we also assume that $0 \in \Lambda$ and that $V(0)=\displaystyle \min_{x  \in \overline{\Lambda}}V(x)$.
\end{itemize}
Hereafter, we consider the sets 
$$
M=\left\{x \in \overline{\Lambda} \,:\, V(x)=V(0)\right\}
$$
and
$$
M_\delta=\left\{ x \in \overline{\Lambda}\,:\, dist(x,M) \leq \delta \right\},
$$
where $\delta>0$ and $dist(x,M)$ denotes the usual distance in $\mathbb{R}^N$ between $x$ and $M$.

Our main result is the following:
\begin{theorem} \label{T01} Suppose that $f$ satisfies  the conditions $(f_1) - (f_3)$ and that $V$ satisfies $(V).$ Then for each 
	$\delta>0$ small enough, there exist $\epsilon_0 > 0$ and $V_*>0$ such that (\ref{11}) admits at least $cat_{M_\delta}(M)$ couples
	$(u_{j}, \lambda_{j})\in H^{1}(\mathbb{R}^N)\times \mathbb{R}$ of weak solutions for $0 < \epsilon < \epsilon_0$ and 
	$|V|_\infty < V_*$ with $\int_{\mathbb{R}^{N}}|u_j|^{2}dx=a^{2}\epsilon^N$,  $\lambda_{j}<0$. Moreover, 
	if $u_\epsilon$ denotes one of these solutions and $\xi_\epsilon$ is the global maximum of $|u_\epsilon|$, then $ \displaystyle	\lim_{\epsilon \to 0}V(\xi_\epsilon)=0.$
\end{theorem}

In order to prove the Theorem \ref{T01}, we will consider the problem below
\begin{align}\label{110}
	\left\{
	\begin{aligned}
		&-\Delta u+V(\epsilon x)u=\lambda u+f(u), \quad
		\quad
		\hbox{in }\mathbb{R}^N,\\
		&\int_{\mathbb{R}^{N}}|u|^{2}dx=a^{2},
	\end{aligned}
	\right.
\end{align}
which is equivalent to the problem (\ref{11}). Our goal in the first moment is to prove the existence of at least  $cat_{M_\delta}(M)$ couples
$(u_{j}, \lambda_{j})\in H^{1}(\mathbb{R}^N)\times \mathbb{R}$ critical points for the energy functional 
$$
I(u)=\frac{1}{2}\int_{\mathbb{R}^N}(|\nabla u|^2 +V(\epsilon x)|u|^2)\,dx-\int_{\mathbb{R}^N}F(u)\,dx,\,\, u\in H^{1}(\mathbb{R}^N),
$$
restricted to the sphere $S(a)$ given in (\ref{Sa2}), when $\epsilon$ is small enough. In the second moment, we study the concentration phenomena when $\epsilon \to 0$.

We would like to point out that if $Y$ is a closed subset of a topological space $X$, the
Lusternik-Schnirelmann category $cat_X(Y)$ is the least number of closed and contractible
sets in $X$ which cover $Y$. If $X = Y$, we use the notation $cat(X)$. For more details about this
subject, we cite \cite{Willem}.

The paper is organized as follows. In Section \ref{section2}, we recall some technique results. In Section 3, we study the penalized problem, more precisely, we study the Palais-Smale condition of modified energy function on the sphere $S(a)$ and prove some tools which are useful to establish a multiplicity result. Finally, in Section 4, we prove the multiplicity and concentration of solutions to problem (\ref{11}).

\section{Some technical results}\label{section2}

In this short section, we recall some results found in \cite{A21} that involve the existence of normalized solution for the problem 
 \begin{equation}\label{ct1}
 	\begin{cases}
 		&	-\Delta u+\mu u=\lambda u+f(u), \quad\hbox{in }\mathbb{R}^N,\\
 		&\int_{\mathbb{R}^{N}}|u|^{2}dx=a^{2},
 	\end{cases}
 \end{equation}
where $a>0$, $\mu \geq 0$ and $\lambda\in \mathbb{R}$ is an unknown parameter that appears as a Lagrange multiplier 
and $f$ is a continuous function satisfying $(f_1)-(f_3)$.   

It is well known that a solution $u$ to the problem (\ref{ct1}) corresponds to a critical point of the following $C^{1}$ functional
$$
I_\mu(u)=\frac{1}{2}\int_{\mathbb{R}^N}(|\nabla u|^2 +\mu|u|^2)\,dx-\int_{\mathbb{R}^N}F(u)\,dx,\,\, u\in H^{1}(\mathbb{R}^N),
$$
restricted to the sphere $S(a)$ given in (\ref{Sa2}), where $F(t)=\int\limits_{0}^{t}f(s)ds.$ 

The main result associated with problem $(\ref{ct1})$ is the following:
\begin{theorem} \label{T12} 
	Suppose that $f$ satisfies  the conditions $(f_1) - (f_3)$. Then, there is $V_*>0$ such that problem (\ref{ct1}) has a couple $(u,\lambda)$ solution for all $\mu \in [0,V_*]$, where $u$ is positive, radial and $\lambda<0$.
\end{theorem}

In the proof of the above theorem, it is proved that the real number
$$
\mathcal{I}_{\mu,a}=\inf_{u \in S(a)}I_\mu(u)
$$
is achieved and $\mathcal{I}_{\mu,a}<0$. A crucial result that was used in the proof of the Theorem \ref{T12} is the compactness theorem on $S(a)$ below, which will be also used in the present paper later on. 
\begin{theorem} \label{310} (Compactness theorem on $S(a)$)~~  Let $(u_n)\subset S(a)$ be a minimizing sequence with respect to $\mathcal{I}_{\mu,a}$ and $\mu \in [0,V_*]$.
	Then, for some subsequence either\\
	\noindent $i)$ \, $(u_n)$ is strongly convergent, \\
	or \\
	\noindent $ii)$ \, There exists $(y_n) \subset \mathbb{R}^N$ with $|y_n|\rightarrow +\infty$ such that the sequence $v_n(x)=u_n(x+y_n)$ is strongly convergent to a function $v\in S(a)$ with  $\mathcal{I}_{\mu}(v)=\mathcal{I}_{\mu,a}$ 
\end{theorem}

An immediate consequence of Theorem \ref{T12} is the following corollary 
\begin{corollary} \label{Cor1} Fix $a>0$ and let $0\le \mu_1<\mu_2 \leq V_*$. Then, $\mathcal{I}_{\mu_1,a}<\mathcal{I}_{\mu_2,a}<0$.
	
\end{corollary} 

\section{The penalized problem}\label{section3}

In what follows, we set $\tau,b>0$ such that $\frac{f(b)}{b^{q-1}}=\tau$ and the function 
\begin{equation*}
	\tilde{f}(t):=
	\begin{cases}
		f(t), & 0 \leq |t| \leq  b,\\
		\tau|t|^{q-2}t, & |t| \geq b,
	\end{cases}
\end{equation*}
where $q$ was fixed in $(f_3)$. Using the function above, we introduce the penalized nonlinearity $g:\mathbb{R}^{N}\times \mathbb{R}\rightarrow \mathbb{R}$ given by 
\begin{equation*} \label{1.7}
		g(x,t):=\chi_{\overline{\Lambda}}(x)f(t)+(1-\chi_{\overline{\Lambda}}(x))\tilde{f}(t),
\end{equation*}
where $\chi_{\overline{\Lambda}}$ is the characteristic function on $\overline{\Lambda}$. 

In view of  $(f_{1})$--$(f_{3})$, we see that $g$ is an odd  Carath\'{e}odory function satisfying the following properties:
\begin{itemize}
	\item[$(g_{1})$]$\underset{t\rightarrow 0}{\lim}\frac{g(x,t)}{t}=0$ uniformly in $x\in \mathbb{R}^{N}$; \\
	
	\item[$(g_{2})$] $|g(x,t)|\leq \tau |t|^{q-1}$, for each $x \in \overline{\Lambda}^{c}$ and $t\in \mathbb{R}$; \\
	
	\item[$(g_{3})$] There is $C>0$ such that $|g(x,t)|\leq C(|t|^{q_0 -1}+|t|^{q-1})$, for all  $ x \in \mathbb{R}^N$ and $t\in \mathbb{R}$; \\
	
	\item[$(g_{4})$] $0 \leq g(x,t) \leq f(t) $, for each $x \in \mathbb{R}^N$ and $t\geq 0,$  and $g(x,t)t\le f(t)t$ for all $x\in\mathbb R^N$ and $t\in\mathbb R;$
	
	\item[$(g_5)$]  $\displaystyle G(x,t):=\int_{0}^{t}g(x,s)ds \leq F(t), \quad \forall x \in \mathbb{R}^N$ and $t \in \mathbb{R}.$
	\end{itemize}
Since $q \geq p_1$, the conditions $(g_1)$ and $(g_3)$ are immediate, and so, we will only show  $(g_2)$, $(g_4)$ and $(g_5)$. We start by showing $(g_2)$. Note that for $x\in {\overline \Lambda}^c$ and $t\ge b,$  
\begin{align}\label{cta1*}
g(x,t)=\tau t^{q-1}=\tau |t|^{q-1}.
\end{align}
On the other hand, for $x\in {\overline \Lambda}^c$ and $t \in (0,b]$, the assumption $(f_3)$ ensures that 
\begin{align}\label{cta2*}
f(t)=\frac{f(t)}{t^{q-1}}t^{q-1}\leq\dfrac{f(b)}{b^{q-1}}t^{q-1}=\tau t^{q-1}.  
\end{align}
The definition of $g$ together with (\ref{cta1*})-(\ref{cta2*}) gives
$$
g(x,t)=f(t) \leq \tau t^{q-1}=\tau |t|^{q-1}, \quad \forall x \in  {\overline \Lambda}^c \quad \mbox{and} \quad t \geq 0.  
$$
For $x\in {\overline \Lambda}^c$ and $t \leq 0$, we use the last inequality together with fact that $g$ is odd to get
$$
|g(x,t)|=|-g(x,-t)|=|g(x,-t)|=g(x,-t) \leq \tau (-t)^{q-1}=\tau |t|^{q-1}, 
$$
showing $(g_2)$.  In order to prove $(g_4)$, note that for $x\in {\overline \Lambda}^c$ and $t\ge b,$ one has   
\begin{align}\label{cta1}
	g(x,t)=\tau t^{q-1}=\dfrac{f(b)}{b^{q-1}}t^{q-1}.
\end{align}
Applying the assumption $(f_3),$  
\begin{align}\label{cta2}
	\dfrac{f(b)}{b^{q-1}}\le\dfrac{f(t)}{t^{q-1}}\;\text{for all}\; t\ge b.
\end{align}
Combining (\ref{cta1}) and (\ref{cta2}), we obtain $g(x,t)\le f(t)$ for all $x\in {\overline \Lambda}^c$ and $t\ge b.$ If $x\in {\overline \Lambda}^c$ and $0\le t\le b,$ we have
$g(x,t)=\tilde f(t)=f(t).$ Now, for all $x\in\overline\Lambda,$ we see that $g(x,t)=f(t),$ and this proves $(g_4)$, because $g(x,t)t$ and $f(t)t$ are odd in the variable $t$. Related to the $(g_5)$, since $g(x,t)$ and $f$ is odd function, then $G(x,t)$ and $F(t)$ are even functions with respect to the variable $t.$ This fact combined with $(g_4)$ gives $(g_5)$.

Now we consider the  modified problem
\begin{align}\label{1100}
	\left\{
	\begin{aligned}
		&-\Delta u+V(\epsilon x)u=\lambda u+g_\epsilon(x,u), \quad
		\quad
		\hbox{in }\mathbb{R}^N,\\
		&\int_{\mathbb{R}^{N}}|u|^{2}dx=a^{2},
	\end{aligned}
	\right.
\end{align}
where $g_\epsilon(x,t)=g(\epsilon x,t)$, for all $x \in \mathbb{R}^N$ and $t \in \mathbb{R}$. 

Note that, if  $u_{\epsilon}$ is a solution of $(\ref{1100})$ with
$$
\vert u_{\epsilon}(x)\vert \leq b \quad\text{for all } x\in \overline{\Lambda}^{c}_{\epsilon},
\quad
\overline{\Lambda}_{\epsilon}:=\{x\in \mathbb{R}^{2}: \epsilon x\in \overline{\Lambda}\},
$$
then $u_{\varepsilon}$ is a solution of \eqref{110}.

A solution $u$ to the problem \eqref{1100} with  $\displaystyle \int_{\mathbb{R}^{N}}|u|^{2}dx=a^{2}$ corresponds to a critical point of the functional
$$
J_\epsilon(u)=\frac{1}{2}\int_{\mathbb{R}^N}(|\nabla u|^2 +V(\epsilon x)|u|^2)\,dx-\int_{\mathbb{R}^N}G_\epsilon(x,u)\,dx,\,\, u\in H^{1}(\mathbb{R}^N),
$$
restricted to the sphere $S(a)$ given in (\ref{Sa2}), where $G_\epsilon(x,t)=\int_{0}^t g_\epsilon(x,s)\,ds$ and $|\,\,\,\,|_p$ denotes the usual norm in $L^{p}(\mathbb{R}^N)$ for $p \in [1,+\infty].$ 

It is well known that  $J_\epsilon \in C^{1}(H^{1}(\mathbb{R}^N),\mathbb{R})$ and 
$$
J'_\epsilon(u)v=\int_{\mathbb{R}^N}(\nabla u \nabla v\,+V(\epsilon x)uv)\,dx-\int_{\mathbb{R}^N}g_\epsilon(x,u)v\,dx, \quad \forall v 
\in H^{1}(\mathbb{R}^N).
$$

\begin{lemma} \label{L1} The functional $J_\epsilon$ is bounded from below in $S(a)$.
	
\end{lemma}
\begin{proof} The proof is standard, but for the reader's convenience we will write some lines of the proof. From $(g_3)$,  
\begin{equation*} \label{EQ01}
	J_\epsilon(u) \geq \frac{1}{2}\int_{\mathbb{R}^N} |\nabla u|^2 \,dx
	-	C\int_{\mathbb{R}^N}(|u|^{q_0}+|u|^{q})\,dx.
\end{equation*}
By the Gagliardo-Nirenberg inequality, 
\begin{equation*}\label{Gagliardo}
	|u|^{l}_{l} \leq C|u|^{(1-\beta_l )l}_{2} \vert \nabla u\vert^{\beta_l l}_{2} , \ \mbox{in}\ \mathbb{R}^N (N\geq 1),\ \beta_l=N(\frac{1}{2}-\frac{1}{l}),
\end{equation*}
for some positive constant $C=C(N, l)>0,$ where $l \in [2,\frac{2N}{N-2})$ if $N \geq 3$ and $l \geq 2$. Hence,
\begin{equation*} \label{EQ1}
		J_\epsilon(u) \geq \frac{1}{2}\int_{\mathbb{R}^N} |\nabla u|^2 \,dx
		- C_qa^{(1-\beta_{q})q}\left(\int_{\mathbb{R}^N}|\nabla u|^2 \,dx \right)^{\frac{\beta_q q}{2} }-
		C_pa^{(1-\beta_{q_0})q_0}\left(\int_{\mathbb{R}^N}|\nabla u|^2 \,dx \right)^{\frac{\beta_{q_0} q_0}{2} }.
\end{equation*}
As $q,q_0 \in (2,2+\frac{4}{N})$, clearly $\beta_{q} q,\beta_{q_0}q_0 <2$, which ensures the boundedness of $J_\epsilon$ from below. 
\end{proof}

The last lemma guarantees that the real number
$$
\Upsilon_{\epsilon,a}=\inf_{u \in S(a)}J_\epsilon(u)
$$
is well defined. Next, we will establish some properties of $\Upsilon_{\epsilon,a}$ that are crucial in our approach. 

\begin{lemma} \label{L1*} There is $V_*>0$ and $\Gamma=\Gamma(V_*)>0$ independent of $\tau>0$ and $\epsilon>0$ such that $\Upsilon_{\epsilon,a}<-\Gamma$ for $|V|_\infty<V_*$ and for all $\tau>0$.
\end{lemma}
\begin{proof} The assumption $(f_1)$ yields that $\displaystyle \lim_{t \to 0^+}\frac{q_0F(t)}{t^{q_0}}=\alpha>0$, then there is $\delta \in (0, b)$ such that 
	\begin{equation} \label{NOVA1}
		\frac{q_0F(t)}{t^{q_0}}\geq \frac{\alpha}{2}, \quad \forall t \in (0, \delta].
	\end{equation}
	Given $u_0 \in S(a) \cap C^{\infty}_{0}(\mathbb{R}^N)$ be a nonnegative function, we set
	\begin{equation*}
		\mathcal{H}(u_0, s)(x)=e^{\frac{Ns}{2}}u_0(e^{s}x),\, \text{for all}\,\, x \,\in \mathbb{R}^N \,\mbox{and all} \,\, s\,\in \mathbb{R}.
	\end{equation*}
	A direct computation provides that for $s<0$ and $|s|$ large enough, we must have  
	$$
	0 \leq e^{\frac{Ns}{2}}u_0(e^{s}x)\leq \delta, \quad \forall x \in\mathbb{R}^N.
	$$
	Then, by definition of $g$, 
	$$
	G_\epsilon(x,\mathcal{H}(u_0, s)(x))=F(\mathcal{H}(u_0, s)(x)), \quad \forall x \in \mathbb{R}^N.
	$$
	Consequently, 
	\begin{equation*} \label{CONV0}
		\int_{\mathbb{R}^N}\vert \mathcal{H}(u_0, s)(x)\vert^{2}\,dx=a^{2}
	\end{equation*}
	and 
	$$
	\int_{\mathbb{R}^N}G_\epsilon(x,\mathcal{H}(u_0, s)(x))\,dx=\int_{\mathbb{R}^N}F(\mathcal{H}(u_0, s)(x))\,dx=e^{-{Ns}}\int_{\mathbb{R}^N}F(e^{\frac{Ns}{2}}u_0(x))\,dx.
	$$
From \eqref{NOVA1},  
	$$
	\int_{\mathbb{R}^N}G(x,\mathcal{H}(u_0, s)(x))\,dx \geq \dfrac{\alpha}{2q_0} e^{\frac{(q_0-2)Ns}{2}}\int_{\mathbb{R}^N}|u_0(x)|^{q_0}\,dx,
	$$
	and so, 
	$$
	J_\epsilon(\mathcal{H}(u_0, s)) \leq \frac{e^{2s}}{2}\int_{\mathbb{R}^N}|\nabla u_0|^2 dx+\dfrac{|V|_\infty a^2}{2}-
	\frac{\alpha e^{\frac{(q_0-2)Ns}{2}}}{2q_0}\int_{\mathbb{R}^N}|u_0(x)|^{q_0}dx.
	$$
	Since $q_0 \in (2,2+\frac{4}{N})$, increasing $|s|$ if necessary, we derive that
	$$
	\frac{e^{2s}}{2}\int_{\mathbb{R}^N}|\nabla u_0|^2 dx-\frac{\alpha e^{\frac{(q_0-2)Ns}{2}}}{2q_0}\int_{\mathbb{R}^N}|u_0(x)|^{q_0}dx=
	A_s<0,
	$$
from where it follows that 
	$$
	J_\epsilon(\mathcal{H}(u_0, s)) \leq \dfrac{|V|_\infty a^2}{2}+ A_s. 
	$$
	Now, we fix $V_*>0$ satisfying
	$$
	A_s+\dfrac{V_* a^2}{2}<0. 
	$$
Thereby, if $|V|_\infty <V_*$ and   $\Gamma= -\left(A_s+\dfrac{V_* a^2}{2}\right)$, one gets
	$$
	J_\epsilon(\mathcal{H}(u_0, s))< -\Gamma,  
	$$
proving the lemma. 
\end{proof}

In what follows, we denote by $I_{0},I_\infty:H^{1}(\mathbb{R}^N) \to \mathbb{R}$ the following functionals
$$
I_{0}(u)=\frac{1}{2}\int_{\mathbb{R}^N}(|\nabla u|^2+V(0)|u|^2)\,dx-\int_{\mathbb{R}^N}F(u) \,dx
$$
and
$$
I_{\infty}(u)=\frac{1}{2}\int_{\mathbb{R}^N}(|\nabla u|^2 +V_\infty|u|^2)\,dx-\int_{\mathbb{R}^N}F(u) \,dx.
$$
Moreover, let us designate by $\mathcal{I}_{0,a}$ and  $\mathcal{I}_{\infty,a}$ the following real numbers
$$
\mathcal{I}_{0,a}=\inf_{u \in S(a)}I_{0}(u) \quad \mbox{and} \quad 
\mathcal{I}_{\infty,a}=\inf_{u \in S(a)}I_\infty(u).
$$
Because $0 < V_\infty=\displaystyle \min_{x \in \partial \Lambda}V(x)<+\infty,$ the Corollary \ref{Cor1} asserts that 
\begin{equation} \label{DES1} 
	\mathcal{I}_{0,a}<\mathcal{I}_{\infty,a}<0.
\end{equation}
In the sequel, we fix $0<\rho_1=\frac{1}{2}(\mathcal{I}_{\infty,a}-\mathcal{I}_{0,a})$.

Our next lemma establishes a crucial relation involving the levels $\Upsilon_{\epsilon,a}, \mathcal{I}_{\infty,a}$ and $\mathcal{I}_{0,a}$.  

\begin{lemma} \label{cepsilon} There is $\epsilon_0>0$ such that
	$\Upsilon_{\epsilon,a} < \mathcal{I}_{\infty,a}$ for all $\epsilon \in (0, \epsilon_0)$. 
	
\end{lemma}
\begin{proof} Let $u_{0} \in S(a)$ with ${I}_{0}(u_{0})=\mathcal{I}_{0,a}$, and consider the function $\tilde{u}_\epsilon(x)=u_0(\frac{\epsilon x-x_0}{\epsilon}) \in S(a)$ with $x_0 \in \Lambda$ and $V(x_0)=V(0)$. A simple calculus gives  
	$$
	\Upsilon_{\epsilon,a}\leq J_\epsilon(\tilde{u}_{\epsilon})=\frac{1}{2}\int_{\mathbb{R}^N}(|\nabla {u}_{0}|^2+V(\epsilon x+x_0)|{u}_{0}|^2) \,dx-\int_{\mathbb{R}^N} G_\epsilon(\epsilon x+x_0,{u}_{0}) \,dx.
	$$
	Letting $\epsilon \to 0^{+}$, one finds
	\begin{equation*} 
		\limsup_{\epsilon \to 0^{+}}\Upsilon_{\epsilon,a}\leq \lim_{\epsilon \to 0^{+}}J_\epsilon({u}_{0})={I}_{0}({u}_{0})=\mathcal{I}_{0,a}.
	\end{equation*}
	The estimate $\Upsilon_{\epsilon,a} < \mathcal{I}_{\infty,a}$, for $\epsilon$ small enough, follows from \eqref{DES1} together with the last inequality. 
	
\end{proof}

The next two lemmas are technical results that will be used to prove the $(PS)_c$ condition for $J_\epsilon$ restricted to $S(a)$ at some levels. From now on, we are fixing $\epsilon \in (0,\epsilon_0),$ where $\varepsilon_0$ is given in Lemma \ref{cepsilon}.

\begin{lemma} \label{L2} There is $\tau_*>0$  independent of $\epsilon \in (0, \epsilon_0)$ such that if $\tau \in (0,\tau_*)$, then  the following property holds: If $(u_n) \subset S(a)$ satisfies $J_\epsilon(u_n) \to c$ with $c<\mathcal{I}_{0,a}+\rho_1<0$ and  $u_n \rightharpoonup u$ in $H^{1}(\mathbb{R}^N)$, then $u\not=0$.
\end{lemma}
\begin{proof} Arguing as in the proof of Lemma \ref{L1} and employing $(g_1)-(g_3)$, we derive that
\begin{equation*} 
	J_\epsilon(w) \geq \frac{1}{2}\int_{\mathbb{R}^N} |\nabla w|^2 \,dx
	-\tau\int_{\Lambda_\epsilon^c}|w|^{q} \,dx-
	C\int_{\Lambda_\epsilon}(|w|^{q_0}+|w|^q) \,dx ,
\end{equation*}
and so, 
\begin{equation*}
	J_\epsilon(w) \geq \frac{1}{2}\int_{\mathbb{R}^N} |\nabla w|^2 \,dx
	-\tau C_qa^{(1-\beta_q)q}\left(\int_{\mathbb{R}^N}|\nabla w|^2 \,dx \right)^{\frac{\beta_q q}{2} }-
	C\int_{\Lambda_\epsilon}(|w|^{q_0}+|w|^q)\,dx, \quad \forall w \in H^{1}(\mathbb{R}^N).
\end{equation*}
	
Now, setting the function $h:[0,+\infty) \to [0,+\infty)$ given by 
\begin{equation} \label{h}
h(t)= \frac{1}{2}t-\tau C_qa^{(1-\beta_q)q}t^{\frac{\beta_q q}{2} }, \quad \forall t \geq 0,
\end{equation}	
and using the fact that $\frac{\beta_q q}{2}<1$, it is possible to check that there is $t_\tau>0$ satisfying
$$
0>h(t_\tau)=\min_{t \geq 0}h(t).
$$
A straightforward computation shows that $h(t_\tau) \to 0$ as $\tau \to 0$. Hence, there is $\tau_*>0$ such that 
\begin{equation} \label{htau}
0>h(t_\tau)>\frac{1}{2}(\mathcal{I}_{\infty,a}+\mathcal{I}_{0,a}), \quad \forall \tau \in (0,\tau_*).
\end{equation}
Assume by contradiction that $u=0$. Then,
$$
\frac{1}{2}(\mathcal{I}_{\infty,a}+\mathcal{I}_{0,a})=\mathcal{I}_{0,a}+\rho_1>c=J_{\epsilon}(u_n)+o_n(1) \geq  h(t_\tau)-C\int_{\Lambda_\epsilon}(|u_n|^{q_0}+|u_n|^{q}) \,dx=h(t_\tau)+o_n(1),
$$
that leads to 
$$
\frac{1}{2}(\mathcal{I}_{\infty,a}+\mathcal{I}_{0,a})> h(t_\tau)+o_n(1).
$$
Letting the limit of $n \to +\infty$, we arrive at 
	$$
\frac{1}{2}(\mathcal{I}_{\infty,a}+\mathcal{I}_{0,a})\geq  h(t_\tau), 
	$$
which contradicts (\ref{htau}). Thus, the weak limit $u$ of $(u_n)$ is nontrivial.
\end{proof}
\begin{lemma} \label{L3*} Let $(u_n) \subset S(a)$ be a $(PS)_c$ sequence for $J_\epsilon$ restricted to $S(a)$ with 
	$c<\mathcal{I}_{0,a}+\rho_1<0$ and $u_n \rightharpoonup u_\epsilon$ in $H^{1}(\mathbb{R}^N)$, that is, 
	\begin{equation*} \label{gamma(a)}
		J_\epsilon(u_n) \to c \quad \mbox{as} \quad n \to +\infty \quad \mbox{and} \quad 
		\|J_\epsilon|'_{S(a)}(u_n)\| \to 0 \quad \mbox{as} \quad n \to +\infty.
	\end{equation*}
	If $v_n=u_n - u_\epsilon \not\to 0$ in $H^{1}(\mathbb{R}^N)$, then decreasing $\epsilon_0$  if necessary, there is $\beta>0$ independent of $\epsilon \in (0,\epsilon_0)$ and $\tau \in (0,\tau^*)$ such that 
	$$
	\liminf_{n \to +\infty}|u_n-u_\epsilon|_{2}^{2}\geq \beta. 
	$$	
\end{lemma}
\begin{proof} Setting the functional $\Psi:H^{1}(\mathbb{R}^N) \to \mathbb{R}$ given by
	$$
	\Psi(u)=\frac{1}{2}\int_{\mathbb{R}^2}|u|^2\,dx,
	$$
	we see that $S(a)=\Psi^{-1}(\{a^2/2\})$. Then, by Willem \cite[Proposition 5.12]{Willem}, there exists 
	$(\lambda_n) \subset \mathbb{R}$ such that
	\begin{equation*} 
		||J_\epsilon'(u_n)-\lambda_n\Psi'(u_n)||_{(H^{1}(\mathbb{R}^N))'} \to 0 \quad \mbox{as} \quad n \to +\infty.
	\end{equation*} 
	Since $(u_n)$ is bounded in $H^{1}(\mathbb{R}^N)$, it follows that $(\lambda_n)$ is also a bounded sequence, 
	then we can assume that $\lambda_n \to \lambda_\epsilon$ as $n \to +\infty$, and thereby, 
	\begin{equation*} \label{EQUAT1*}
	||J_\epsilon'(u_n)-\lambda_\epsilon\Psi'(u_n)||_{(H^{1}(\mathbb{R}^N))'} \to 0 \quad \mbox{as} \quad n \to +\infty
\end{equation*} 
and 
	$$
	J_\epsilon'(u_\epsilon)-\lambda_\epsilon\Psi'(u_\epsilon)=0 \quad \mbox{in} \quad (H^{1}(\mathbb{R}^N))'.
	$$
Arguing as in \cite[Lemma 2.8-iv)]{AVT}, it is possible to prove that 
$$
J_\epsilon'(u_n)=J_\epsilon'(u_\epsilon)+J_\epsilon'(v_n)+o_n(1)
$$
and
$$
\Psi_\epsilon'(u_n)=\Psi_\epsilon'(u_\epsilon)+\Psi_\epsilon'(v_n)+o_n(1).
$$
From this,
$$
J_\epsilon'(u_n)-\lambda_\epsilon \Psi'(u_n)= J_\epsilon'(u_\epsilon)-\lambda_\epsilon \Psi'(u_\epsilon) + J_\epsilon'(v_n)-\lambda_\epsilon \Psi'(v_n)+o_n(1)=J_\epsilon'(v_n)-\lambda_\epsilon \Psi'(v_n)+o_n(1)
$$
leading to 
\begin{equation*} \label{EQUAT3}
	||J_\epsilon'(v_n)-\lambda_\epsilon \Psi'(v_n)||_{(H^{1}(\mathbb{R}^N))'} \to 0 \quad \mbox{as} \quad n \to +\infty.
\end{equation*}
By $(f_3)$, we know that $qF(t)\leq f(t)t$ for all $ t\geq 0$. Therefore,  
	$$
	0>\rho_1+\mathcal{I}_{0,a} \geq \liminf_{n \to +\infty}J_\epsilon(u_n)=\liminf_{n \to +\infty}\left(J_\epsilon(u_n)-\frac{1}{2}J'_\epsilon(u_n)u_n+\frac{1}{2}\lambda_na^2\right) \geq \frac{1}{2}\lambda_\epsilon a^2,
	$$
	implying that  
	$$
	\limsup_{\epsilon \to 0}\lambda_\epsilon \leq \frac{2(\rho_1+\mathcal{I}_{0,a})}{a^2}<0.
	$$
	Thereby, there is $\lambda_*>0$ independent of $\epsilon$ and $\tau$ such that 
	$$
	\lambda_\epsilon \leq -\lambda_*<0, \quad \forall \epsilon \in (0,\epsilon_0).
	$$
	The above analysis yields in the equality below 
	$$
	\int_{\mathbb{R}^N}(|\nabla v_n|^2+V(\epsilon x)|v_n|^2)\,dx-\lambda_\epsilon \int_{\mathbb{R}^N}|v_n|^2\,dx=\int_{\mathbb{R}^N}g_\epsilon(x,v_n)v_n\,dx+o_n(1),
	$$
which combines with $(g_1)$ and $(g_4)$ to give
	$$
	\int_{\mathbb{R}^N}(|\nabla v_n|^2+C_0|v_n|^2)\,dx \leq C_2|v_n|_{p}^{p}+o_n(1),
	$$
for some positive constant $C_0>0$ that does not depend on $\epsilon \in (0,\epsilon_0)$ and $p \in (2,2+\frac{4}{N})$. By using the continuous Sobolev embedding $H^{1}(\mathbb{R}^N) \hookrightarrow L^{p}(\mathbb{R}^N)$, one gets 
\begin{equation} \label{NOVAp}
	\|v_n\|^2_{H^{1}(\mathbb{R}^N)} \leq C_3|v_n|_{p}^{p}+o_n(1) \leq C_4\|v_n\|^p_{H^{1}(\mathbb{R}^N)}+o_n(1),
\end{equation}
where $C_3,C_4>0$ are independent of $\epsilon$. Since $v_n \not\to 0$ in $H^{1}(\mathbb{R}^N)$, for some subsequence of $(v_n)$, still denoted by $(v_n)$, we can assume that $\displaystyle \liminf_{n \to +\infty}\|v_n\|_{H^{1}(\mathbb{R}^N)}>0$. Have this in mind, the last inequality ensures that 
\begin{equation} \label{EQUAT4}
	\liminf_{n \to +\infty}\|v_n\|_{H^{1}(\mathbb{R}^N)} \geq \left(\frac{1}{C_4}\right)^{\frac{1}{p-2}}.
\end{equation}
From (\ref{NOVAp}) and (\ref{EQUAT4}), 
\begin{equation*} \label{EQUAT4*}
	\liminf_{n \to +\infty}	|v_n|^{p}_{p} \geq C_5,
\end{equation*}
for some $C_5>0$ that does not depend on $\epsilon$.

Employing again the Gagliardo-Nirenberg inequality, we arrive at 
	\begin{equation} \label{EQUAT5}
		\liminf_{n \to +\infty}	|v_n|^{p}_{p} \leq C(\liminf_{n \to +\infty}	|v_n|_{2})^{(1-\beta_p )p} K^{\beta_p p}  \quad \mbox{with} \quad \beta_p=N(\frac{1}{2}-\frac{1}{p}), 
	\end{equation}
	where $K>0$ is a constant independent of $\epsilon \in (0, \epsilon_0)$ satisfying $|\nabla v_n|_2\leq K$ for all $n \in\mathbb{N}$. Now the lemma follows from \eqref{EQUAT4} and \eqref{EQUAT5}.
		
\end{proof}
From now on, we fix $0<\rho<\min\{\frac{1}{2},\frac{\beta}{a^2}\}(\mathcal{I}_{\infty,a}-\mathcal{I}_{0,a})\leq \rho_1.$

\begin{lemma} \label{ps-d} Decreasing if necessary $\tau^*>0$, for each $\epsilon \in (0,\epsilon_0)$ the functional $J_\epsilon$ satisfies the $(PS)_c$ condition restricted
	to $S(a)$ for $c < \mathcal{I}_{0,a}+\rho.$
	
\end{lemma}
\begin{proof} Assume that $u_n \rightharpoonup u$ in $H^{1}(\mathbb{R}^N)$ and fix $v_n=u_n-u$. If $v_n \not\to 0$ in $H^{1}(\mathbb{R}^N)$, we can employ the Lemma \ref{L3*} to conclude that 
\begin{equation} \label{NOVAEQ}
\liminf_{n \to +\infty}|v_n|_{2}^{2} \geq \beta,
\end{equation}
for some $\beta>0$ independent of $\epsilon \in (0,\epsilon_0)$ and 
\begin{equation*} \label{EQUAT1*}
		||J_\epsilon'(v_n)-\lambda_n\Psi'(v_n)||_{(H^{1}(\mathbb{R}^N))'} \to 0 \quad \mbox{as} \quad n \to +\infty.
\end{equation*}
From $(f_3)$,  we have that $q F(t)\leq f(t)t$ for all $ t\geq 0$, and so,  
$$
0>\rho_1+\mathcal{I}_{0,a} \geq \liminf_{n \to +\infty}J_\epsilon(u_n)=\liminf_{n \to +\infty}\left(J_\epsilon(u_n)-\frac{1}{2}J'_\epsilon(u_n)u_n+\frac{\lambda_na^2}{2}\right) \geq \frac{\lambda_\epsilon a^2}{2}
$$
where $\lambda_n \to \lambda_\epsilon$. Hence,   
$$
\limsup_{\epsilon \to 0}\lambda_\epsilon \leq \frac{2(\rho_1+\mathcal{I}_{0,a})}{a^2}=\frac{1}{a^2}(\mathcal{I}_{\infty,a}+\mathcal{I}_{0,a})<0.
$$
Thus, decreasing if necessary $\epsilon_0$, there is $\lambda_*>0$ independent of $\epsilon$ and $\tau$ such that 
$$
\lambda_\epsilon \leq -\lambda_*<0, \quad \forall \epsilon \in (0,\epsilon_0).
$$
The above analysis ensures that 
$$
\int_{\mathbb{R}^N}(|\nabla v_n|^2+V(\epsilon x)|v_n|^2)\,dx-\lambda_\epsilon \int_{\mathbb{R}^N}|v_n|^2\,dx=\int_{\mathbb{R}^N}g_\epsilon(x,v_n)v_n\,dx+o_n(1),
$$
and so, 
$$
\int_{\mathbb{R}^N}|\nabla v_n|^2\,dx+\lambda_* \int_{\mathbb{R}^N}|v_n|^2\,dx \leq \tau C_qa^{(1-\beta_q)q}\left(\int_{\mathbb{R}^N}|\nabla v_n|^2 \,dx \right)^{\frac{\beta_q q}{2} }+\int_{\Lambda_\epsilon}f(v_n)v_n\,dx+o_n(1)
$$
which leads to
$$
\begin{array}{l}
h(t_\tau)+\lambda_*\displaystyle \int_{\mathbb{R}^N}|v_n|^2\,dx \leq  \int_{\mathbb{R}^N}|\nabla v_n|^2\,dx-\tau C_qa^{(1-\beta_q)q}\left(\int_{\mathbb{R}^N}|\nabla v_n|^2 \,dx \right)^{\frac{\beta_q q}{2}}+\lambda_*\int_{\mathbb{R}^N}|v_n|^2\,dx  \\
\mbox{} \\
\hspace{4 cm} \le \displaystyle \int_{\Lambda_\epsilon}f(v_n)v_n\,dx+o_n(1),
\end{array}
$$
where $h(t)$ is given in (\ref{h}). Taking the limit of $n \to \infty$ and using (\ref{NOVAEQ}), we obtain
$$
h(t_\tau)+\lambda_*\beta\leq 0, \quad \forall \tau \in (0,\tau_*). 
$$
Since $h(t_\tau) \to 0$ as $\tau \to 0$, the inequality above does not hold when $\tau$ is small enough, because $\lambda_*\beta>0$ and it does not depend on $\epsilon \in (0,\epsilon_0)$ and $\tau \in (0, \tau^*)$. Therefore, decreasing $\tau^*$ if necessary , we deduce that $v_n \to 0$ in $H^{1}(\mathbb{R}^N)$ for $\tau \in (0,\tau^*)$. This proves the desired result. 	
	
\end{proof}

\section{Multiplicity result}

Let $\delta>0$ be fixed and $w$ be a positive solution of the problem
$$
\left\{
\begin{aligned}
	&-\Delta u+V(0)u=\lambda u+f(u), \quad
	\quad
	\hbox{in }\mathbb{R}^N,\\
	&\int_{\mathbb{R}^{N}}|u|^{2}dx=a^{2},
\end{aligned}
\right.
\leqno{(P_0)}
$$
with $\mathcal{I}_{0}(w)=\mathcal{I}_{0,a}$. Let $\eta$ be a smooth nonincreasing cut-off function
defined in $[0,\infty)$ such that $\eta(s)=1$ if $0\le s \leq \frac{\delta}{2}$ and $\eta(s)= 0$ if $s \geq \delta$. For any $y \in M$, let us 
define
$$
\Psi_{\epsilon,y}(x) = \eta(|\epsilon x-y|)w({(\epsilon x-y)}/{\epsilon}),
$$
$$
\tilde{\Psi}_{\epsilon,y}(x)=a\frac{\Psi_{\epsilon,y}(x)}{|\Psi_{\epsilon,y}|_2},
$$
and $\Phi_\epsilon:M \to S(a)$ by $\Phi_\epsilon(y)= \tilde{\Psi}_{\epsilon,y}$. By the definition of $\Phi_\epsilon(y)$, we see that it has compact
support for any $y \in M$. 

\begin{lemma} \label{lemmaC1}
	The function $\Phi_\epsilon$  has the following limit
	$$
	\lim_{\epsilon \to 0}J_\epsilon(\Phi_{\epsilon}(y))=\mathcal{I}_{0,a}, \mbox{uniformly in} \quad y \in M.
	$$	
\end{lemma}
\begin{proof} Suppose by contradiction that the lemma is false. Then there exists $\delta_0>0$, $(y_n) \subset M$ with $y_n \to y \in M$ and $\epsilon_n \to 0$ such that
	$$
	|J_{\epsilon_n}(\Phi_{\epsilon_n}(y_n))-\mathcal{I}_{0,a}|\geq \delta_0, \quad \forall n \in \mathbb{N}.
	$$
	From Lebesgue’s theorem, it follows that 
	$$
	\lim_{n \to +\infty}\int_{\mathbb{R}^N}|\Psi_{\epsilon,y_n}|^2\,dx
	=\lim_{n \to +\infty}\int_{\mathbb{R}^N}|\eta(\epsilon_n z)w(z)|^2\,dx=\int_{\mathbb{R}^N}|w|^2\,dx,
	$$
	$$
	\lim_{n \to +\infty}\int_{\mathbb{R}^N}G(\epsilon_n x, \Phi_{\epsilon_n}(y_n))\,dx
	=\lim_{n \to +\infty}\int_{\mathbb{R}^N}G\left(\epsilon_n z+y_n,a\frac{\eta(\epsilon_n z)w(z)}{|\Psi_{\epsilon_n,y_n}|_2}\right)\,dx=\int_{\mathbb{R}^N}F(w)\,dx,
	$$
	$$
	\lim_{n \to +\infty}\int_{\mathbb{R}^N}|\nabla \Phi_{\epsilon_n}(y_n)|^2\,dx
	=\lim_{n \to +\infty}\int_{\mathbb{R}^N}\dfrac{a^2}{|\Psi_{\epsilon_n,y}|_2^2}|\nabla (\eta(\epsilon_n z)w(z))|^2\,dx=\int_{\mathbb{R}^N}|\nabla w|^2\,dz
	$$
	and
	$$
	\lim_{n \to +\infty}\int_{\mathbb{R}^N}V(\epsilon_n x)|\Phi_{\epsilon_n}(y_n)|^2\,dx
	=V(0)\int_{\mathbb{R}^N}|w|^2\,dx.
	$$
	Consequently,  
	$$
	\lim_{n \to +\infty}J_{\epsilon_n}(\Phi_{\epsilon_n}(y_n))={I}_{0}(w)=\mathcal{I}_{0,a},
	$$
	which is absurd.
\end{proof}

For any $\delta>0$, let $R=R(\delta)>0$ be such that $M_\delta \subset B_R(0)$. Let $\chi: \mathbb{R}^N \to \mathbb{R}^N$ 
be defined as $\chi(x)=x$ for $|x| \leq R$ and $\chi(x)=\frac{R x}{|x|}$ for $|x| \geq R$. Finally, let us consider
$\beta_\epsilon:S(a) \to \mathbb{R}^N$ given by 
$$
\beta_\epsilon(u)=\frac{\int_{\mathbb{R}^N}\chi(\epsilon x)|u|^2\,dx}{a^2}.
$$

\begin{lemma} \label{lemmaC2} The function $\Phi_\epsilon$ has the following limit
	$$
	\lim_{\epsilon \to 0}\beta_\epsilon(\Phi_\epsilon(y))=y, \mbox{uniformly in} \quad y \in M. 
	$$	
\end{lemma}
\begin{proof} Suppose by contradiction that the lemma is false. Then there exist $\delta_0>0$, $(y_n) \subset M$ with $y_n \to y \in M$ and $\epsilon_n \to 0$ such that
	\begin{equation} \label{EQC1}
		|\beta_{\epsilon_n}(\Phi_{\epsilon_n}(y_n))-y_n|\geq \delta_0, \quad \forall n \in \mathbb{N}.
	\end{equation}
	Using the definition of $\Phi_{\epsilon_n}(y_n)$ and $\beta_{\epsilon_n}$, we have the equality below
	$$
	\beta_{\epsilon_n}(\Phi_{\epsilon_n}(y_n))=y_n+\frac{\int_{\mathbb{R}^N}(\chi(\epsilon_n z+y_n)-y_n)|\eta(\epsilon_n z)w(z)|^2\,dz}{a^2}.
	$$
	Using the fact that $(y_n) \subset M \subset B_R(0)$, the Lebesgue’s theorem ensures that
	$$
	|\beta_{\epsilon_n}(\Phi_{\epsilon_n}(y_n))-y_n| \to 0, \quad \mbox{as} \quad n \to +\infty,
	$$
	which contradicts (\ref{EQC1}) and the lemma is proved.
\end{proof}

Let $\gamma:[0,+\infty) \to [0,+\infty)$ be a positive function such that $\gamma(\epsilon) \to 0$ as $\epsilon \to 0$ and let 
\begin{equation} \label{Sa3}
	\tilde{S}(a)=\left\{u \in S(a) : J_\epsilon(u) \leq \mathcal{I}_{0,a} + \gamma(\epsilon)\right\}.
\end{equation} 
Thanks to Lemma \ref{lemmaC1}, the function $\gamma(\epsilon)=\sup_{y \in M}|J_\epsilon(\Phi_\epsilon(y))-\mathcal{I}_{0,a}|$ satisfies
 $\gamma(\epsilon) \to 0$ as $\epsilon \to 0$. Hence, $\Phi_{\epsilon}(y) \in \tilde{S}(a)$ for all $y \in M$. 

\begin{proposition} \label{P1} Let $\epsilon_n \to 0$ and $(u_n) \subset S(a)$ with $J_{\epsilon_n}(u_n) \to \mathcal{I}_{0,a}$ and 	$\|J_{\epsilon_n}|'_{S(a)}(u_n)\| =o_n(1)$ for all $n \in \mathbb{N}$. 
	Then, there is $(\tilde{y}_n) \subset \mathbb{R}^N$ such that $\tilde{u}_n(x)=u_n(x+\tilde y_n)$ has a convergent subsequence in 
	$H^{1}(\mathbb{R}^N)$. Moreover, up to a subsequence, $y_n=\epsilon_n \tilde{y}_n \to y$ for some $y \in M$.   
\end{proposition}
\begin{proof} First of all, we claim that there are $R_0,\tau>0$ and $\tilde{y}_n \in \mathbb{R}^N$ such that
	$$
	\int_{B_{R_0}(\tilde{y}_n)}|u_n|^2\,dx \geq \tau, \quad \forall n \in \mathbb{N}.
	$$
	Otherwise, by Lions' lemma, we must have $u_n \to 0$ in $L^{p}(\mathbb{R}^N)$ for all $p \in (2,2^*)$, leading to $\displaystyle \int_{\mathbb{R}^N}F(u_n)\,dx \to 0$, and so, $\displaystyle \lim_{n \to +\infty}J_{\epsilon_n}(u_n) \geq 0$, which is absurd because $\displaystyle \lim_{n \to +\infty}J_{\epsilon_n}(u_n)=\mathcal{I}_{0,a}<0$. 
	Thus, fixing $\tilde{u}_n(x)=u_n(x+\tilde y_n)$, there is $\tilde{u} \in H^{1}(\mathbb{R}^N) \setminus \{0\}$ such that, up to a subsequence, we may assume that $\tilde{u}_n \rightharpoonup \tilde{u}$ in $H^{1}(\mathbb{R}^N)$. Since $(\tilde{u}_n) \subset S(a)$ and $J_{\epsilon_n}(u_n)\geq I_{0}(u_n)=I_0(\tilde{u}_n) \geq \mathcal{I}_{0,a}$, it follows that   $I_0(\tilde{u}_n) \to \mathcal{I}_{0,a}$. From Theorem \ref{310}, $\tilde{u}_n \to \tilde{u}$ in $H^{1}(\mathbb{R}^N)$, $\tilde{u} \in S(a)$ and 	$\|I_0|'_{S(a)}(\tilde{u})\|=0$. 
	
	\begin{claim} \label{A1} $\underset{n\rightarrow\infty}{\lim}\text{dist}\big(\epsilon_{n}y_{n}, \overline{\Lambda}\big)=0$.\\
		
	\end{claim}
	
	Indeed, if the claim does not hold, there exist $\delta>0$ and a subsequence of $(\epsilon_{n}y_{n})$, still denoted by itself, such that
	$$
	\text{dist}\big(\epsilon_{n}y_{n}, \overline{\Lambda}\big)\geq \delta, \quad \forall n\in \mathbb{N}.
	$$		
	Consequently, there exists $r>0$ such that
	$$
	B_{r}(\epsilon_{n}y_{n})\subset \Lambda^{c},  \quad \forall n\in \mathbb{N}.
	$$	
	Using the fact that $	\|J_{\epsilon_n}|'_{S(a)}(u_n)\| =0$, one has
	$$
	\displaystyle\int_{\mathbb{R}^{N}} |\nabla \tilde{u}_n |^2 +\int_{\mathbb{R}^{N}} V(\epsilon_n x +\epsilon_{n}y_n)|\tilde{u}_n|^2 dx=\lambda_n \int_{\mathbb{R}^N}|\tilde{u}_n|^2\,dx+
	\displaystyle\int_{\mathbb{R}^{N}}  g(\epsilon_n x +\epsilon_{n}y_n, \tilde{u}_{n})\tilde{u}_{n}\,dx+o_n(1), 
	$$
	where $(\lambda_n) \subset \mathbb{R}$ satisfies 
	$$
	\limsup_{n \to +\infty}\lambda_n <0. 
	$$
	Thus, there exist $\lambda^*<0$ and $n_0 \in \mathbb{N}$ such that
	$$
	\lambda_n \leq \lambda^*, \quad \forall n \geq n_0. 
	$$
 Using the equality below
	$$
	\displaystyle\int_{\mathbb{R}^{N}}  g(\epsilon_n x +\epsilon_{n}y_n, \tilde{u}_{n})\tilde{u}_{n}\,dx
	=\int_{B_{\frac{r}{\epsilon_{n}}}(0)} g(\epsilon_n x +\epsilon_{n}y_n, \tilde{u}_{n})\tilde{u}_{n}\,dx+ \int_{\mathbb{R}^N\backslash B_{\frac{r}{\epsilon_{n}}}(0)} g(\epsilon_n x +\epsilon_{n}y_n, \tilde{u}_{n})\tilde{u}_{n}\,dx,
	$$
 it follows that 
	\begin{equation*}
		\displaystyle\int_{\mathbb{R}^{N}} g(\epsilon_n x +\epsilon_{n}y_n, \tilde{u}_{n})\tilde{u}_{n}\,dx \leq \tau \int_{B_{\frac{r}{\epsilon_{n}}}(0)}|\tilde{u}_{n}|^{q}\, dx+ \int_{\mathbb{R}^N\backslash B_{\frac{r}{\epsilon_{n}}}(0)} |f(\tilde{u}_{n})||\tilde{u}_{n}|dx.
	\end{equation*}
	Thereby,
	$$
		\displaystyle\int_{\mathbb{R}^{N}} | \nabla \tilde{u}_n |^2\,dx -\lambda_* \int_{\mathbb{R}^N}|\tilde{u}_n|^2\,dx
	\leq \tau\int_{B_{\frac{r}{\epsilon_{n}}}(0)}|\tilde{u}_{n}|^{q}\, dx+ \displaystyle\int_{\mathbb{R}^N\backslash B_{\frac{r}{\epsilon_{n}}}(0)} |f(\tilde{u}_{n})||\tilde{u}_{n}|dx.
	$$
 Now, as $\tilde{u}_{n} \to \tilde{u}$ in $H^{1}(\mathbb{R}^N)$, 
$$
\int_{B_{\frac{r}{\epsilon_{n}}}(0)}|\tilde{u}_{n}|^{q} \,dx  \to \int_{\mathbb{R}^N}|\tilde{u}|^{q}\, dx \quad\text{as}\quad n\rightarrow\infty
$$
and	
$$
\int_{\mathbb{R}^N\backslash B_{\frac{r}{\epsilon_{n}}}(0)}| f(\tilde{u}_{n})||\tilde{u}_{n}|dx\rightarrow 0\quad\text{as}\quad n\rightarrow\infty.
$$
From this,   
	\begin{eqnarray*}
		\displaystyle\int_{\mathbb{R}^{N}} | \nabla \tilde{u} |^2\,dx -\lambda_* \int_{\mathbb{R}^N}|\tilde{u}|^2\,dx\leq \tau \int_{\mathbb{R}^N}|\tilde{u}|^{q} \, dx,
	\end{eqnarray*}
which combined with Gagliardo-Nirenberg inequality leads to
	\begin{eqnarray*}
		\displaystyle\int_{\mathbb{R}^{N}}  |\nabla \tilde{u}|^2 \,dx -\lambda_* \int_{\mathbb{R}^N}|\tilde{u}|^2\,dx \leq  \tau C_qa^{(1-\beta_q)q}\left(\int_{\mathbb{R}^N}|\nabla \tilde{u}|^2 \,dx \right)^{\frac{\beta_q q}{2}},
	\end{eqnarray*}
and so, 
	\begin{eqnarray*}
		h(t_\tau)-\lambda_* a^2 \leq 	\displaystyle\int_{\mathbb{R}^{N}}  |\nabla \tilde{u}|^2 \,dx -C_qa^{(1-\beta_q)q}\left(\int_{\mathbb{R}^N}|\nabla \tilde{u}|^2 \,dx \right)^{\frac{\beta_q q}{2}}-\lambda_* \int_{\mathbb{R}^N}|\tilde{u}|^2\,dx \leq 0,
	\end{eqnarray*}
	where $h(t)$ was given in (\ref{h}). Since $\lambda_*<0$ does not depend on $\tau$ and $h(\tau) \to 0$ when $\tau \to 0$, decreasing $\tau^*$ if necessary, it follows that $h(t_\tau)-\lambda_* a^2>0$, which is absurd and the Claim 3.1 is proved.
	
	\begin{claim} \label{A2} $y_{0}\in \Lambda$.		
	\end{claim}	
	First of all, by (${g_5}$), 
	$$
	\begin{array}{l}
		\mathcal{I}_{0,a}+o_n(1)=J_{\epsilon_n}(u_n)=\displaystyle \frac{1}{2}\int_{\mathbb{R}^N}(|\nabla \tilde{u}_n|^2 +V(\epsilon_n x+\epsilon_n y_n)|\tilde{u}_n|^2)\,dx-\int_{\mathbb{R}^N}G(\epsilon_n x+\epsilon_n y_n,\tilde{u}_n)\,dx \\
		\mbox{} \\
		\hspace{4 cm} \geq \displaystyle \frac{1}{2}\int_{\mathbb{R}^N}(|\nabla \tilde{u}_n|^2 +V(\epsilon_n x+\epsilon_n y_n)|\tilde{u}_n|^2)\,dx-\int_{\mathbb{R}^N}F(\tilde{u}_n)\,dx.
	\end{array}	
	$$
	Taking the limit of $n \to +\infty$ in the above inequality, we arrive at  
	$$
	\mathcal{I}_{0,a} \geq \frac{1}{2}\int_{\mathbb{R}^N}(|\nabla \tilde{u}|^2 +V(y_0)|\tilde{u}|^2)\,dx-\int_{\mathbb{R}^N}F(\tilde{u})\,dx.
	$$
	If $y_0 \in \partial \Lambda$, we must have $V(y_0) \geq V_\infty$, from where it follows that
	$$
	\frac{1}{2}\int_{\mathbb{R}^N}(|\nabla \tilde{u}|^2 +V(y_0)|\tilde{u}|^2)\,dx-\int_{\mathbb{R}^N}F(\tilde{u})\,dx \geq \mathcal{I}_{V_\infty,a}>\mathcal{I}_{0,a},
	$$
	which is a contradiction, and so, we must have $y_0 \in \Lambda$. \end{proof}

\begin{lemma} \label{lemmaC3} Let $\delta>0$ and $M_{\delta}=\left\{x \in \overline{\Lambda} \,:\,dist(x,M) \leq \delta \right\}$. Then, 
	$$
	\lim_{\epsilon \to 0}\sup_{u \in \tilde{S}(a) } \inf_{z \in M_\delta}|\beta_\epsilon(u)-z|=0.
	$$
\end{lemma}
\begin{proof} Let $\epsilon_n \to 0$ and $u_n \in \tilde{S}(a)$ such that 
	$$
	\inf_{z \in M_\delta}|\beta_{\epsilon_n}(u_n)-z|=\sup_{u \in \tilde{S}(a) } \inf_{z \in M_\delta}|\beta_{\epsilon_n}(u_n)-z|+o_n(1).
	$$
From the equality above, it suffices to find a sequence $(z_n) \subset M_\delta$ such that
	$$
	\lim_{n \to +\infty}|\beta_\epsilon(u_n)-z_n|=0.
	$$
	Since $u_n \in \tilde{S}(a)$, 
	$$
	\mathcal{I}_{0,a}\leq I_0(u_n) \leq J_{\epsilon_n}(u_n)\leq 	\mathcal{I}_{0,a}+\gamma(\epsilon_n), \quad \forall n \in \mathbb{N},
	$$ 
	and so,  
	$$
	u_n \in {S}(a) \quad \mbox{and} \quad J_{\epsilon_n}(u_n) \to 	\mathcal{I}_{0,a}.
	$$
	From Proposition \ref{P1}, there is $(\tilde{z}_n) \subset \mathbb{R}^N$ such that $z_n=\epsilon_n \tilde{z}_n \to y$ for some 
	$z \in M$ and $\tilde{u}_n(x)=u_n(x+\tilde z_n)$ is strongly convergent to some $ \tilde{u} \in H^{1}(\mathbb{R}^N) \setminus \{0\}$. Then, $(z_n) \subset M_\delta$ for $n$ large enough and 
	$$
	\beta_{\epsilon_n}(u_n)=z_n+ \frac{\int_{\mathbb{R}^N}(\chi(\epsilon_n z+z_n)-z_n)|\tilde{u}_n|^2\,dz}{a^2},
	$$
	implying that
	$$
	\beta_{\epsilon_n}(u_n)-z_n=\frac{\int_{\mathbb{R}^N}(\chi(\epsilon_n z+z_n)-z_n)|\tilde{u}_n|^2\,dz}{a^2} \to 0 \quad \mbox{as} \quad n \to +\infty,
	$$
	proving the lemma. 
\end{proof}
\subsection{Proof of Theorem \ref{T01}} We will divide the proof into three parts: \\

\vspace{0.1 cm}

\noindent {\bf Part I: Multiplicity of solutions for problem (\ref{1100}).} \\

Hereafter, we set $\epsilon \in (0, \epsilon_0)$. Then, by Lemmas 
\ref{lemmaC1}, \ref{lemmaC2} and \ref{lemmaC3}, we can argue as in \cite{CingolaniLazzo} to conclude that $\beta_\epsilon \circ \Phi_\epsilon$ is
homotopic to the inclusion map $id:M \to M_\delta$, and so, 
$$
cat(\tilde{S}(a)) \geq cat_{M_\delta}(M).
$$
By Lemmas \ref{L1} and \ref{ps-d}, we know that $J_{\epsilon}$ is bounded from below on $S(a)$ and $J_\epsilon$ satisfies the $(PS)_c$ condition for $c \in (\mathcal{I}_{0,a},\mathcal{I}_{0,a}+\gamma(\epsilon))$ respectively. Then, the
Lusternik-Schnirelmann category of critical points (see \cite{Ghoussoub} and \cite{Willem} ) gives that  $J_\epsilon$ 
has at last $cat_{M_\delta}(M)$ of critical points on $S(a)$. 

\vspace{0.3 cm}

\noindent {\bf Part II: Multiplicity of solutions for problem (\ref{110}).} \\

Let $u_\epsilon$ be a solution of (\ref{1100}) with $J_\epsilon(u_\epsilon) \leq \mathcal{I}_{0,a}+\gamma(\epsilon)$, where $\gamma$ was given in (\ref{Sa3}).  Our goal is to prove that decreasing $\epsilon_0$ if necessary, we have  
$$
\vert u_{\epsilon}(x)\vert \leq b \quad\text{for all } x\in \overline{\Lambda}^{c}_{\epsilon},
\quad \overline{\Lambda}_{\epsilon}:=\{x\in \mathbb{R}^{2}: \epsilon x\in \overline{\Lambda}\},
$$
showing that $u_\epsilon$ is a solution of (\ref{110}) for all $\epsilon \in (0, \epsilon_0)$.

 Arguing as in the proof of Proposition \ref{P1}, for each $\epsilon_n \to 0$, there is $\tilde{y}_n \in \mathbb{R}^N$ such that $y_n=\epsilon_n\tilde{y}_n \to y$ with $y \in M$ and $\tilde{u}_n(x)=u_n(x+\tilde{y}_n)$ is strongly convergent to $\tilde{u} \in H^{1}(\mathbb{R}^N)$ with $\tilde{u} \not=0$. As $\tilde{u}_n$ is a solution of 
$$
-\Delta \tilde{u}_n+V(\epsilon_n x+y_n)\tilde{u}_n=\lambda_n \tilde{u}_n +f(\tilde{u}_n), \quad \mbox{in} \quad \mathbb{R}^N, 
$$
with 
$$
\limsup_{\epsilon \to 0}\lambda_n \leq \frac{2(\rho_1+\mathcal{I}_{0,a})}{a^2}<0,
$$
the convergence $\tilde{u}_n \to \tilde{u}$ in $H^{1}(\mathbb{R}^N)$ permits to apply the same arguments found in \cite[Lemma 4.5]{AFig} to conclude that  
$$
\lim_{|x| \to +\infty}\tilde{u}_n(x)=0, \quad \mbox{uniformly in} \quad \mathbb{N}.
$$
Therefore, given $\tau>0$, there are $R_1>0$ and  $n_0 \in \mathbb{N}$ such that
$$
|\tilde{u}_n(x)| \leq \frac{b}{2} \quad \mbox{for} \quad |x| \geq R_1 \quad \mbox{and} \quad n \geq n_0.
$$ 
Now, using the fact that $y_n=\epsilon_n \tilde{y}_n \to y_0 \in \Lambda$, there is $n_0 \in \mathbb{N}$ such that if $x \not\in \Lambda_{\epsilon_n}$, then  
$$
|x-\tilde{y}_n|=\frac{1}{\epsilon_n}|\epsilon_n x-y_n|\geq \frac{r_1}{2\epsilon_n}>R_1, \quad \forall n \geq n_0,
$$
where $r_1=dist(y_0,\partial \Lambda)$. Hence, 
$$
|u_n(x)|=|\tilde{u}_n(x-\tilde{y}_n)|\leq \frac{b}{2}, \quad \forall x \not\in \Lambda_{\epsilon_n} \quad \mbox{and} \quad n \geq n_0, 
$$
showing the $u_n$ is a solution of (\ref{110}) for $n \geq n_0$. This proves the desired result. 

\vspace{0.3 cm}

\noindent {\bf Part III: Concentration phenomena  of the solutions for problem (\ref{110}).} \\

The concentration phenomena of the solutions follows as in \cite[Subsection 4.1, See Part II]{AT1} and we omit its proof. 

\vspace{1 cm}

\noindent {\bf Data Availability Statement:} Data sharing not applicable to this article as no datasets were generated or analysed during the current study.

\end{document}